\documentclass{article}
\usepackage{geometry}
\usepackage{titling}
\AtEndDocument{\bigskip{\footnotesize%
\textsc{Department of Applied Mathematics, Delhi Technological University, Delhi–110042, India} \par  
  \textit{E-mail address:} \texttt{suryagiri456@gmail.com} \par
  \addvspace{\medskipamount}
  \textsc{Department of Applied Mathematics, Delhi Technological University, Delhi–110042, India} \par
  \textit{E-mail address:} \texttt{spkumar@dtu.ac.in} \par
}}
\usepackage{doc}

\usepackage{url}

\usepackage{graphicx}
\usepackage{epstopdf}
\usepackage{amsthm}
\usepackage{amssymb}
\usepackage{amsmath}
\usepackage{enumitem}

\DeclareMathOperator{\RE}{Re}

\usepackage{array}
\usepackage{enumitem}
\usepackage[utf8]{inputenc}
\usepackage[english]{babel}
\newtheorem{theorem}{Theorem}
\newtheorem{corollary}{Corollary}[theorem]
\newtheorem{lemma}[theorem]{Lemma}

\newtheorem*{remark}{Remark}
\newtheorem*{assumption}{Assumption}

\usepackage{varwidth}
\usepackage{lineno, hyperref}
\usepackage{graphicx}
\usepackage{epstopdf}
\usepackage{amsmath}
\usepackage{amsthm}
\usepackage{amsthm}
\usepackage{float}
\usepackage{subcaption}
\usepackage{caption}
\usepackage{authblk}
\usepackage{abstract}

\begin{document}


\title{Sharp Estimate of Fifth Coefficient  for Ma Minda Starlike and Convex Functions}
\author{Surya Giri and S. Sivaprasad Kumar}
\date{}

\maketitle





\begin{abstract}
    \noindent  In the present investigation, we find the sharp bound of fifth coefficient of analytic normalized function $f$ satisfying $z f'(z)/f(z) \prec \varphi(z)$ when coefficients of $\varphi$ satisfy certain conditions. For an appropriate choice of $\varphi$, the already known estimates for various other subclasses of starlike functions follow directly from the obtained result.
\end{abstract}
\vspace{0.5cm}





%

\section{Introduction}\label{sec1}

Let $\mathcal{A}$ be the class of analytic functions of the form
    $ f(z)=z+\sum_{n=2}^\infty a_n z^{n} ,$
   in the unit disk $\mathbb{D}=\{ z \in \mathbb{C}: \vert z \vert < 1 \} $.
   Suppose $\overline{\mathbb{D}}=\{ z\in\mathbb{C}: \vert z \vert \leq 1 \}$ and $\mathbb{T}=\{ z\in\mathbb{C}: \vert z \vert=1 \}$. Further, we denote the subclass of $\mathcal{A}$ containing univalent functions by $\mathcal{S}$. In 1916, Bieberbach conjectured that $\vert a_n \vert \leq n$ for $f\in \mathcal{S}$, which was settled in 1985 by L. D. Branges (see~\cite[Page 5]{VAllu}). During this time period, the conjecture was verified for various other subclasses of $\mathcal{S}$. Since  the growth, covering and distortion theorems for functions $f\in \mathcal{S}$ can be proved using the fact $\lvert  a_2 \lvert  \leq 2,$ so the coefficient bound plays a major role in identifying the geometric nature of the function.

   Undoubtedly, the primary and extensively studied subclasses of $\mathcal{S}$ are the classes of starlike and convex functions, respectively denoted by $\mathcal{S}^*$ and $\mathcal{C}$. Analytically, a function $f\in \mathcal{S}^*$ if and only if $\RE (z f'(z))/f(z))>0$ and $f\in \mathcal{C}$ if and only if $\RE (1 + (z f''(z))/f'(z))>0$ for $z\in \mathbb{D}$.

    An analytic function $f$ is subordinate to another analytic function $g$, if there exists a Schwarz function  $\omega$ such that   $f(z)=g(\omega(z))$ for all $z\in \mathbb{D}$ simply denoted by $f \prec g$. By $\mathcal{B}_0$, we represent the class of Schwarz functions having the form
\begin{equation}\label{omega}
    \omega(z)=\sum_{n=1}^\infty c_n z^n.
\end{equation}
  Using the concept of subordination, Ma and Minda \cite{1} introduced the classes
\begin{equation*}
    \mathcal{S}^*(\varphi)=\bigg\{ f\in\mathcal{A} : \frac{z f'(z)}{f(z)}\prec \varphi(z)\bigg\} \\
\end{equation*}
   and
\begin{equation*}\label{varphi(z)}
  \mathcal{C}(\varphi)=\bigg\{ f\in\mathcal{A} : 1 + \frac{z f''(z)}{f'(z)}\prec \varphi(z)\bigg\},
\end{equation*}
    where $\varphi$ is an analytic univalent function in $\mathbb{D}$ satisfying (i) $\RE \varphi(z)>0$  for $z\in \mathbb{D}$, (ii) $\varphi'(0)>0$, (iii) $\varphi(\mathbb{D})$ is a starlike domain with respect to $\varphi(0)=1$ and (iv) $\varphi(\mathbb{D})$ is symmetric about the real axis. Suppose $\varphi$ has the following series expansion
\begin{equation}\label{ma-minda}
    \varphi(z)=1+B_1 z+B_2 z^2+B_3 z^3+\cdots, \quad B_1>0.
\end{equation}
     Since $\varphi(\mathbb{D})$ is symmetric with respect to the real axis and $\varphi(0)=1$, we have $\overline{\varphi{(\bar{z})}}=\varphi(z)$, which yields that all $B_i$'s are real. For the family $\mathcal{S}^*(\varphi)$ and $\mathcal{C}(\varphi)$ sharp growth, distortion theorems and   estimates for the coefficient functional $\lvert  a_3 -r a_2^2 \rvert $ are known, where $r\in\mathbb{R}$ \cite{1}. In fact, these classes  unify several   subfamilies of starlike and convex functions. For instance, if
     $\varphi(z)=(1+(1-2\alpha)z)/(1-z)$, the classes $\mathcal{S}^*(\varphi)$ and $\mathcal{C}(\varphi)$ reduce to the classes of starlike functions of order $\alpha$ and convex functions of order $\alpha$, denoted by $\mathcal{S}^*(\alpha)$ and $\mathcal{C}(\alpha)$ respectively ($0\leq\alpha < 1 $)~(see~\cite{8}).  In case of $\alpha=0$, we simply obtain $\mathcal{S}^*(0) := \mathcal{S}^*$ and $\mathcal{C}(0) := \mathcal{C}$. If we define $\varphi(z) = \sqrt{1+z}$, then the class $\mathcal{S}^*(\varphi)$ coincides with the class $\mathcal{S}^*_L$ introduced by Sok\'{o}\l\ and Stankiewicz~\cite{21}. Geometrically, a function $f\in \mathcal{S}^*_L$ if and only if $z f'(z)/f(z)$ lies in the region bounded by the right lemniscate of Bernoulli given by $\vert w^2 -1 \vert < 1.$ Therefore $\mathcal{S}^*_L = \{ f \in \mathcal{A} : \vert (zf'(z)/f(z))^2 -1 \vert < 1\}.$ Mendiratta et al.~\cite{31} considered the class $\mathcal{S}^*_{RL}$ of functions $f$ such that the quantity $z f'(z)/f(z)$ lies in the interior of the left half of the shifted lemniscate of Bernoulli given by $\Omega_{RL} =\{ w \in \mathbb{C}: \RE w>0, \vert ( w - \sqrt{2})^2 - 1 \vert < 1$. Note that, the function $\varphi(z)= \sqrt{2}-(\sqrt{2}-1)\sqrt{( 1 - z)/( 1 + 2 (\sqrt{2} - 1 ) z)}$ maps the unit disk onto $\Omega_{RL}$. Thus
     $$ \mathcal{S}^*_{RL} = \bigg\{ f\in \mathcal{A} : \frac{z f'(z)}{f(z)} \prec   \sqrt{2}-(\sqrt{2}-1)\sqrt{\frac{1-z}{1+2(\sqrt{2}-1)z}} \bigg\} .$$
     Using this approach, various interesting subclasses of starlike functions by confining the values of $zf'(z)/f(z)$ to a defined region within the right half-plane  were introduced and studied.
     Some of them are listed in  table 1 along with their respective class notations.
\begin{table}[ht]
  \label{tb1}
  \caption{Subclasses of starlike functions} 
  \centering 
  \begin{tabular}{lll} 
  \hline 
    {\bf{Class}}  & \textbf{$\varphi(z)$} &  \textbf{Reference}      \\ [0.5ex] 
    \hline 
      $\mathcal{S}^*_{\sin}$                &  $ 1 + \sin{z}$  &  \cite{Cho2} Cho et al.   \\
     $\mathcal{S}^*_{SG}$    &   $2/(1+e^{-z})$   &  \cite{9}  Goel and Kumar   \\
     $\mathcal{S}^*_{\varrho}$   &   $1+z e^z$   &   \cite{18}   Kumar and  Gangania   \\
     $\mathcal{S}^*_{q_b}$   &   $\sqrt{1+b z}, \quad b\in(0,1]$  & \cite{28}   Sok\'{o}\l   \\
    \hline 
  \end{tabular}
\end{table}

      Note that the sharp bounds for the initial coefficients $\lvert a_2\rvert ,$   $\lvert  a_3 \rvert $ and $\lvert  a_4 \rvert $ of functions in  $\mathcal{S}^*(\varphi)$ for a general choice of $\varphi$ have been established~(see~\cite{12,38}), however, for $\lvert  a_5\rvert$, it is still open.  Sok\'{o}\l \ \cite{10} conjectured that $\lvert a_n \rvert\leq 1/(2(n-1))$ for $f\in \mathcal{S}^*_L$ and Mendiratta et al. \cite{31} conjectured that $\lvert a_n \rvert \leq (5-3 \sqrt{2})/(2(n-1))$ for $f\in \mathcal{S}^*_{RL}$. Ravichandran and Verma \cite{5} settled these conjectures for $n=5$. For different subclasses of $\mathcal{S}^*$ depending on the different choices of $\varphi(z)$ in $\mathcal{S}^*(\varphi)$, the bound for $\lvert a_5\rvert$ is known \cite{19,9,18}. For $n\geq 5$, finding the bound of $\lvert a_n\rvert$ for functions belonging to the classes $\mathcal{S}^*(\varphi)$ and $\mathcal{C}(\varphi)$ in case of general $\varphi$ is still an open problem \cite{OP}. In this study, we  obtain sharp bound of fifth coefficient for functions in $S^*(\varphi)$ and $\mathcal{C}(\varphi)$ under specific conditions on the coefficients of $\varphi$. The obtained results give several new special cases and some already known results as special cases.

 We require the following lemmas  to prove our results. Let   $\mathcal{P}$ be the   Carath\'{e}odory class containing   analytic functions  of the form
\begin{equation}\label{p(z)}
    p(z)=1+\sum_{n=1}^\infty p_n z^n, \quad z\in \mathbb{D}
\end{equation}
    with $\RE{p(z)}>0$.
    Clearly, the function
\begin{equation}\label{L(z)}
    L (z)=\frac{1+z}{1-z}, \quad z\in\mathbb{D}
\end{equation}
    is a member of $\mathcal{P}$ as it maps the   unit disk on to the right half-plane.
   \begin{lemma} \label{L2}
    \cite[Lemma I]{23} If the functions $1+\sum_{n=1}^\infty b_n z^n$ and $1+\sum_{n=1}^\infty c_n z^n$ are in $\mathcal{P}$, then the same holds for the function
       $$1+\frac{1}{2}\sum_{n=1}^\infty b_n c_n z^n.$$
\end{lemma}

\begin{lemma}\label{A_4}
    \cite[Lemma II]{23} Let $h(z)=1+u_1 z+u_2 z^2+\cdots$ and $1+G(z)=1+d_1 z+d_2 z^2+\cdots$ be functions in $\mathcal{P}$, and set
     $$\gamma_n=\frac{1}{2^n}\left[ 1 + \frac{1}{2}\sum_{k=1}^n{n \choose k}u_k \right], \quad \gamma_0=1.$$
     If $A_n$ is defined by
     $$\sum_{n=1}^\infty (-1)^{n+1}\gamma_{n-1}G^n (z)=\sum_{n=1}^\infty A_n z^n,$$
     then $\lvert A_n\rvert\leq 2.$
\end{lemma}

      It is worth recalling the M\"obius function $\Psi_{\zeta}$ which maps the unit disk onto the unit disk and given by
\begin{equation}\label{MobiousT}
    \Psi_{\zeta}(z)=\frac{z-\zeta}{1-\overline{\zeta}z}, \quad \zeta\in \mathbb{D}.
\end{equation}

\begin{lemma}\label{Cho's lemma}
    \cite[Lemma 2.4]{24} If $p \in\mathcal{P}$, then for some $\zeta_i\in \overline{\mathbb{D}},i\in\{1,2,3\}$,
\begin{equation}\label{p_1}
    p_1=2\zeta_1,
\end{equation}
\begin{equation}\label{p_2}
    p_2=2\zeta_1^2+2(1- \lvert \zeta_1 \rvert^2)\zeta_2,
\end{equation}
\begin{equation}\label{p_3}
    p_3=2\zeta_1^3+4(1- \lvert \zeta_1 \rvert^2)\zeta_1 \zeta_2-2(1- \lvert \zeta_1 \rvert^2)\overline{\zeta_1}\zeta_2^2+2(1- \lvert \zeta_1 \rvert^2)(1- \lvert \zeta_2 \rvert^2)\zeta_3.
\end{equation}

For $\zeta_1,\zeta_2\in\mathbb{D}$ and $\zeta_3\in\mathbb{T}$, there is a unique function $p=L \circ \omega \in\mathcal{P}$ with $p_1,$ $p_2$ and $p_3$ as in (\ref{p_1})-(\ref{p_3}), where
\begin{equation}\label{omega_2(z)}
    \omega(z)=z\Psi_{-\zeta_1}(z\Psi_{-\zeta_2}(\zeta_3 z)),\quad z\in\mathbb{D},
\end{equation}
    that is
    $$p(z)=\frac{1+(\overline{\zeta_2}\zeta_3+\overline{\zeta_1}\zeta_2+\zeta_1)z+(\overline{\zeta_1}\zeta_3+\zeta_1\overline{\zeta_2}\zeta_3+\zeta_2)z^2+\zeta_3 z^3}{1+(\overline{\zeta_2}\zeta_3+\overline{\zeta_1}\zeta_2-\zeta_1)z+(\overline{\zeta_1}\zeta_3-\zeta_1\overline{\zeta_2}\zeta_3-\zeta_2)z^2-\zeta_3 z^3},\quad z\in\mathbb{D}.$$
\end{lemma}
      Conversely, if $\zeta_1, \zeta_2 \in \mathbb{D}$ and $\zeta_3\in\overline{\mathbb{D}}$ are given, then we can construct a (unique) function $p\in\mathcal{P}$ of the form (\ref{p(z)}) so that $p_i, i\in\left\{ 1,2,3\right\} $, satisfy the identities in (\ref{p_1})-(\ref{p_3}). For this, we define
\begin{equation}\label{omega_3(z)}
    \omega(z)=\omega_{\zeta_1, \zeta_2, \zeta_3} (z)= z \Psi_{-\zeta_1}(z \Psi_{-\zeta_2}(\zeta_3 z)), \quad z\in\mathbb{D},
\end{equation}
    where $\Psi_{\zeta}$ is the function given as in (\ref{MobiousT}). Then $\omega \in \mathcal{B}_0$. Moreover, if we define $p(z)=(1+\omega(z))/(1-\omega(z)), z\in\mathbb{D}$, then $p$ is represented by (\ref{p(z)}), where $p_1$, $p_2$ and $p_3$ satisfy the identities in (\ref{p_1})-(\ref{p_3}) (see the proof of \cite[Lemma 2.4]{24}).

\section{Estimation of the Fifth Coefficient }\label{fifth coeff}

    We begin with the following lemma:
\begin{lemma}\label{M L}
      If $-1< \sigma <1$, then $F(z)=(1+2 \sigma z+ z^2)/(1-z^2)$ belongs to $\mathcal{P}$.
\end{lemma}
\begin{proof} Let us consider
             $$\omega(z)=\frac{F(z)-1}{F(z)+1}= \frac{z (z+\sigma)}{1+\sigma z}.$$
             From (\ref{MobiousT}), we have
             $$\omega(z)= z \Psi_{-\sigma} (z) , \quad z \in \mathbb{D}.$$
             Since $\Psi_{-\sigma}(z)$ is a conformal automorphism of $\mathbb{D}$, which gives $\lvert  w(z)\rvert <1$ and $w(0)=0$. Therefore $w$ is a Schwarz function and $F\in \mathcal{P}$.
\end{proof}
   To prove our next result, we require the following assumption:
\begin{assumption}
   Let $\varphi(z)$ be   as defined in (\ref{ma-minda}), whose coefficients satisfy the following conditions:
\begin{equation*}
\begin{aligned}
  {\bf{C1}}: \; &   \;\;\;\;\;\;\;\;\;\; \;\;\;\;\;\;  \;\;\;\;\;\; \;\;\;\;\;\;\;\;\;\;\;\; \lvert B_1^2 + 2 B_2 \rvert < 4 B_1^2,  \\
  {\bf C2}:\; & \lvert B_1^3 - B_1^2 B_2 + 18 B_2^2 - 18 B_1 B_3\rvert <3 \lvert ( B_1^2 + 2 B_1 + 2 B_2) (2 B_1^2 -3 B_1  + 3 B_2)\rvert, \\
   {\bf C3}:\; & \lvert 30 B_1^7 - 9 B_1^8 - B_1^6 ( 66 B_2 - 5 ) - 648 B_2^3 + 324 B_2^4 + B_1^5 ( 170 B_2 -126 ) \\
     &- 648 B_2 B_3^2 + B_1^3 (-180 B_2 + 220 B_2^2 + 108 B_3 - 360 B_2 B_3) + B_1 (1296 B_2 B_3 \\
     & - 720 B_2^2 B_3) + 648 B_2^2 B_4 + B_1^4 (108 + 10 B_2 - 175 B_2^2 + 90 B_3 + 162 B_4) \\
     & + B_1^2 ( -144 B_2^2 + 4 B_2^3 + 180 B_2 B_3 - 324 B_3^2 - 648 B_4 + 648 B_2 B_4)\rvert <  8 \lvert 9 B_1^6 \\
     &+ 9 B_1^7 + B_1^4 (-27 + 32 B_2) + B_1^5 (-52 + 63 B_2) + 162 B_2^2 B_3 + B_1^3 (81 - 189 B_2 \\
     & + 164 B_2^2 + 9 B_3) + B_1^2 (18 B_2^2 - 9 B_2 B_3) + B_1 (-162 B_2^2 + 198 B_2^3 - 81 B_3^2)\rvert, \\
     {\bf{C4}}:\; & \;\;\; \;\;\;\;\;\; \;\;\;\;\;\; 0<  (4 B_1^2 + 6( B_2- B_1)/((3 B_1^2+ 6 (B_2-  B_1)) < 1.
\end{aligned}
\end{equation*}
\end{assumption}

\begin{theorem}\label{main}  Let $\varphi(z)$ be   as defined in (\ref{ma-minda}), whose coefficients satisfy the above conditions \textbf{C1} to \textbf{C4}.
   If $f\in\mathcal{S}^*(\varphi)$, then
  $$ \lvert a_5 \rvert \leq \frac{ B_1 }{4}.$$
      The inequality is sharp.
\end{theorem}
\begin{proof}
     Suppose $f\in\mathcal{S}^*(\varphi)$, then
      $$\frac{zf'(z)}{f(z)}=\varphi(\omega(z)),$$
      where $\omega\in \mathcal{B}_0$. If we choose $\omega(z)=(p(z)-1)/(p(z)+1)$, where $p\in\mathcal{P}$ given by (\ref{p(z)}), then
     by comparing the coefficients obtained by series expansion of $f(z)$ together with $p(z)$ and $\varphi(z)$ yields
\begin{equation}\label{a_5=I}
    a_5=\frac{B_1}{8}I,
\end{equation}
where
\begin{equation}\label{I}
    I=p_4+I_1 p_1^4+ I_2 p_1^2 p_2+ I_3 p_1 p_3 + I_4 p_2^2,
\end{equation}
      with
\begin{equation}\label{I1}
      \begin{aligned}
      I_1 &= \frac{1}{48 B_1}\bigg( B_1^4 - 6 B_1^3 + 11 B_1^2 +6 B_1^2 B_2 - 6 B_1 + 3 B_2^2 - 22 B_1 B_2 + 18 B_2  \\
                                                                      & - 18 B_3 + 8 B_1 B_3 + 6 B_4\bigg),\\
       \end{aligned}
\end{equation}
\begin{equation}\label{I2}
    I_2 = \frac{3 B_1^3 - 11 B_1^2 + 9 B_1 - 18 B_2 + 11 B_1 B_2 + 9 B_3}{12 B_1}, \quad I_3 =\frac{2 B_1^2  -3 B_1 + 3 B_2}{3 B_1}
\end{equation}
      and
\begin{equation}\label{I3}
       I_4 = \frac{ B_1^2 - 2 B_1  + 2 B_2}{4 B_1}.
\end{equation}  
  Let  $q(z)=1+b_1 z+b_2 z^2+\cdots$ be in $\mathcal{P}$, then by Lemma \ref{L2}, we have
\begin{equation*}
      1+\frac{1}{2}(p(z)-1)*(q(z)-1)=1+\frac{1}{2}\sum_{n=1}^\infty b_n p_n z^n \in \mathcal{P}.
\end{equation*}
      If we assume $h(z)=1+\sum_{n=1}^\infty u_n z^n \in \mathcal{P}$ and take $1+G(z):=1+\frac{1}{2}\sum_{n=1}^\infty b_n p_n z^n$, then Lemma \ref{A_4} gives
       $$\lvert A_4\rvert \leq 2,$$
     where
\begin{equation}\label{eq:16}
    A_4=\frac{1}{2}\gamma_0 b_4 p_4-\frac{1}{4}\gamma_1 b_2^2 p_2^2-\frac{1}{2}\gamma_1 b_1 b_3 p_1 p_3+\frac{3}{8}\gamma_2 b_1^2 b_2 p_1^2 p_2-\frac{1}{16}\gamma_3 b_1^4 p_1^4
\end{equation}
      and for $i\in\{0,1,2,3 \}$, $\gamma_i$'s are given by $\gamma_0=1$,
\begin{equation}\label{gamma_i's}
    \gamma_1=\frac{1}{2}\left(1+\frac{1}{2}u_1\right), \quad \gamma_2=\frac{1}{4}\left(1+u_1+\frac{1}{2}u_2\right), \quad \gamma_3=\frac{1}{8}\left(1+\frac{3}{2}u_1+\frac{3}{2}u_2+\frac{1}{2}u_3\right).
\end{equation}
    So from (\ref{I}) and (\ref{eq:16}), we can observe that if there exist $q, h \in\mathcal{P}$ such that
\begin{align*}
    b_4= 2, \quad  I_1 &= -\frac{1}{128} \left(1+\frac{3}{2}u_1+ \frac{3}{2}u_2+\frac{1}{2} u_3\right) b_1^4, \quad I_2 = \frac{3}{32}\left( 1+u_1+\frac{u_2}{2}\right) b_1^2 b_2,\\
    I_3&=-\frac{1}{4} (1+\frac{u_1}{2})b_1 b_3 \quad \text{and} \quad I_4= -\frac{1}{8} \left( 1+\frac{u_1}{2}\right)b_2^2,
\end{align*}
   then we have
\begin{equation}\label{A_4=I eq:18}
     I= A_4.
\end{equation}
The bound for $\lvert  A_4 \rvert $ can be obtained from  Lemma \ref{A_4}, consequently, we can estimate the bound for $\lvert  I\rvert $ and thus we arrive at the desired bound by using (\ref{a_5=I}).
     To prove the theorem, we  construct the functions $q$ and $h$ in such a way that we obtain (\ref{A_4=I eq:18}).\par
    From Lemma \ref{Cho's lemma}, suppose that the functions $q$ and $h$ are constructed by taking $\zeta_1,\zeta_2\in\mathbb{D}$, $\zeta_3\in\overline{\mathbb{D}}$ and $\xi_1, \xi_2 \in \mathbb{D}$, $\xi_3\in\overline{\mathbb{D}}$ respectively, as follows:
\begin{equation}\label{Pr_1}
    q=L \circ \omega_1 \quad \text{and} \quad  h=L\circ \omega_2,
\end{equation}
   where
\begin{equation}\label{Pr_2}
    \omega_1(z)= z \Psi_{-\zeta_1}(z \Psi_{-\zeta_2}(\zeta_3 z)), \quad  \omega_2(z)= z \Psi_{-\xi_1}(z \Psi_{-\xi_2}(\xi_3 z))
\end{equation}
    and $L(z)$ is given by (\ref{L(z)}). So again from Lemma \ref{Cho's lemma}, the $b_i$'s and $u_i$'s, $i\in\left\{ 1,2,3\right\}$ are given by
\begin{align*}
     b_1 &= 2\zeta_1, \quad  b_2 = 2 \zeta_1^2 + 2(1-\lvert \zeta_1\rvert ^2)\zeta_2,\\
    b_3 &= 2\zeta_1^3 +4(1-\lvert \zeta_1\rvert ^2)\zeta_1 \zeta_2 - 2(1-\lvert \zeta_1\rvert ^2) \overline{\zeta_1} \zeta_2^2 + 2(1-\lvert \zeta_1\rvert ^2)(1-\lvert \zeta_2\rvert ^2) \zeta_3
\end{align*}
    and
\begin{align*}
     u_1 &= 2\xi_1, \quad u_2 = 2 \xi_1^2 + 2(1-\lvert \xi_1\rvert ^2)\xi_2,\\
    u_3 &= 2\xi_1^3 +4(1-\lvert \xi_1\rvert ^2)\xi_1 \xi_2 - 2(1-\lvert \xi_1\rvert ^2) \overline{\xi_1} \xi_2^2 + 2(1-\lvert \xi_1\rvert ^2)(1-\lvert \xi_2\rvert ^2) \xi_3.\\
\end{align*}
  There may be many solutions for the above set of equations. For our purpose, we impose some restrictions on the parameters. We take all $\xi_i \in \mathbb{R}$, then
\begin{equation}\label{u_i's}
\begin{aligned}
     u_1 &= 2\xi_1, \quad u_2 = 2 \xi_1^2 + 2(1- \xi_1^2)\xi_2,\\
    u_3 &= 2\xi_1^3 +4(1- \xi_1^2)\xi_1 \xi_2 - 2(1-\xi_1^2) \xi_1 \xi_2^2 + 2(1-\xi_1^2)(1-\xi_2^2) \xi_3.
\end{aligned}
\end{equation}
  Further, if we define
\begin{equation*}\label{zeta_i's}
\begin{aligned}
      \xi_1 &= -\frac{B_1^2 + 2 B_2}{2 B_1}, \quad
      \xi_2 = \frac{B_1^3 - B_1^2 B_2 + 18 B_2^2 - 18 B_1 B_3}{3 \left(B_1^2+ 2 B_1 + 2 B_2 \right) \left(2 B_1^2 -3 B_1  + 3 B_2 \right)},\\
    \xi_3 = &\bigg(- 9 B_1^8 + 30 B_1^7  - B_1^6 \left(66 B_2- 5 \right) + 2 B_1^5 \left( 85 B_2 - 63 \right) +4 B_1^3 (5 B_2 ( 11 B_2 \\
            &- 18 B_3 -9 ) +27 B_3 ) +4 B_1^2 \left( B_2^3 -36 B_2^2 - 81 B_3^2  + 45 B_2 B_3 + 162 \left( B_2 -1 \right) B_4 \right)\\
            &- 144 B_1 \left(5 B_2 - 9 \right) B_2 B_3 +324 B_2 (-2 B_3^2 + B_2 \left( \left(B_2 -2  \right) B_2 + 2 B_4\right) )\\
            & + 18 B_1^4 \left( 9 B_4 + 5 B_3 + 6 \right) - 5 B_1^4 B_2 ( 35 B_2 -2 ) \bigg) \bigg/ \bigg(8 (3 B_1^4 + 2 B_1^3 + 18 B_2^2 \\
            &+ B_1^2 (10 B_2 - 9 ) - 9 B_1 B_3) (B_1 ( 3 B_1^2 + B_1  + 11 B_2 -9) + 9 B_3) \bigg),
\end{aligned}
\end{equation*}
  then the conditions \textbf{C1}, \textbf{C2} and \textbf{C3} on the coefficients $B_1, B_2, B_3$ and $B_4$ yield
      $$\lvert \xi_1\rvert <1, \quad \lvert \xi_2\rvert <1, \quad \lvert \xi_3\rvert <1$$
  respectively.
  Using these $\xi_i$'s in (\ref{u_i's}), we can obtain $u_i$'s, which in turn by using (\ref{gamma_i's}) gives
\begin{equation}\label{gamma_i's eq:22}
\begin{aligned}
    \gamma_1  =& \frac{1}{4} \left(2 - B_1 - \frac{2 B_2}{B_1}\right),\\
     \gamma_2 =&\frac{\left( B_1^2 + 2 B_2 -2 B_1 \right) \left(3 B_1^3 - 11 B_1^2 + B_1 \left( 11 B_2 + 9 \right) +
   9 \left( B_3 - 2 B_2 \right) \right)}{24 B_1 \left(  2 B_1^2 + 3 B_2 -3 B_1 \right)},\\
   \gamma_3 =&-\frac{1}{64 B_1 \left(2 B_1^2 + 3 B_2 - 3 B_1 \right)^2}\bigg(3 \left( B_1^2 + 2 B_2 -2 B_1  \right)^2 ( B_1^4 -6 B_1^3 + B_1^2 (6 B_2 \\
             &+ 11 )  + B_1 \left( 8 B_3 - 22 B_2 -6 \right) + 3 \left( B_2^2 + 6 B_2  - 6 B_3 + 2 B_4 \right) )\bigg).
\end{aligned}
\end{equation}
   Let us consider
\begin{equation*}
    q(z)=\frac{1+2 \sigma z+z^2}{1-z^2},
\end{equation*}
    with $\sigma=\sqrt{(4 B_1^2 + 6( B_2- B_1)/((3 B_1^2+ 6 (B_2-  B_1))}$, then
\begin{equation}\label{b_i's}
    b_1=b_3=2\sigma \quad \text{and}\quad  b_2=b_4=2.
\end{equation}
    If we choose $B_1$ and $B_2$ such that $0<\sigma <1$, which is equivalent to condition \textbf{C4}. Then by Lemma~\ref{M L}, we have  $q\in\mathcal{P}$.
     On putting the values of $b_i's$ and $\gamma_i's$  obtained  from (\ref{gamma_i's eq:22})  and (\ref{b_i's}) respectively in (\ref{eq:16}), we get (\ref{A_4=I eq:18}), which together with (\ref{a_5=I}) gives the desired bound for $\lvert a_5\rvert $.

     Let the function $ H :\mathbb{D}\longrightarrow \mathbb{C}$ be given by
     $$ H(z)=z \exp{\int_0^{z} \frac{\varphi(t^4)-1}{t}dt}=z+ \frac{B_1}{4}z^5 + \frac{1}{32}(B_1^2+ 4 B_2)z^9+\cdots,$$
     where coefficients of $\varphi(z)$ satisfy the conditions \textbf{C1} to \textbf{C4}. Then $H(0)=0$, $H'(0)=1$ and $z H'(z)/H(z)=\varphi(z^4)$ and hence the function $H\in \mathcal{S}^*(\varphi)$, proving the result to be sharp for the function $H$.\end{proof}
    Choose $\varphi(z)=1+\sin{z}$, whose coefficients satisfy \textbf{C1} to \textbf{C4}, to obtain  the following result:
\begin{corollary}
      If $f \in\mathcal{S}^*_{\sin}$, where $\varphi(z)=1+\sin{z}$, then
     $$\lvert a_5\rvert \leq \frac{1}{4}.$$
     The bound is sharp.
\end{corollary}
   For all below mentioned choices of $\varphi$ conditions \textbf{C1} to \textbf{C4} are valid. Therefore, the  bounds of $\lvert a_5\rvert $ for some of the known classes are obtained from our result as a special case.
\begin{remark}
\begin{enumerate}
    \item  If $\varphi(z)=2/(1+e^{-z})$, then $f\in \mathcal{S}^*_{SG}$ and $\lvert a_5\rvert \leq 1/8$ \cite[Theorem\; 4.1]{9}.
    \item  If $\varphi(z)=\sqrt{1+z}$, then $f\in \mathcal{S}^*_L$ and $\lvert  a_5 \rvert  \leq 1/8$ \cite[Theorem\; 3.1]{5}.
    \item  If $\varphi(z)=\sqrt{1+b z}$, then $f\in \mathcal{S}^*_{q_b}$, and $\lvert a_5\rvert \leq b/8$, where $b\in (0,1]$ \cite[Theorem\; 3.1]{19}.
    \item  If $f\in\mathcal{S}^*_{RL}$, then $\lvert  a_5 \rvert \leq \left(5-3 \sqrt{2}\right)/8$ \cite[Theorem\; 3.1]{5}.
    \item  If $\varphi(z)=((1+z)/(1-z))^\delta $, then the conditions of Theorem \ref{main} are satisfied only for $0<\delta\leq \delta_0 \approx 0.350162$. Therefore, $\lvert  a_5 \rvert \leq \delta/2$  for  $f\in S^*(\varphi)$ where $0<\delta \leq \delta_0$ \cite{39}.
\end{enumerate}
\end{remark}
\begin{theorem} Let $\varphi(z)$ be   as defined in (\ref{ma-minda}), whose coefficients satisfy the conditions \textbf{C1} to \textbf{C4}.
   If $f\in\mathcal{C}(\varphi)$, then
  $$ \lvert a_5 \rvert \leq \frac{ B_1 }{20}.$$
      The inequality is sharp.
\end{theorem}
\begin{proof}
       Since $f\in \mathcal{C}(\varphi)$, therefore
\begin{equation}\label{new}
       1+\frac{z f''(z)}{f'(z)}=\varphi ({(p(z)-1)}/{(p(z)+1)} ),
\end{equation}
      where $p \in \mathcal{P}$ is given by (\ref{p(z)}). By comparison of the coefficients of $z$, $z^2$, $z^3$ in (\ref{new}) with the series expansion of $f$, $\varphi$ and $p$, we get
\begin{equation}\label{new2}
       a_5=\frac{B_1}{40}I,
\end{equation}
     where
      $$  I=p_4 + I_1 p_1^4 + I_2 p_1^2 p_2 + I_3 p_1 p_3 + I_4 p_2^2, $$
      with $I_1$, $I_2$, $I_3$ and $I_4$ given as in (\ref{I1}), (\ref{I2}) and (\ref{I3}). Using the same method as in Theorem \ref{main}, we obtain
      $$\lvert I \rvert \leq 2,$$
      when $B_1$, $B_2$, $B_3$ and $B_4$ satisfy all the conditions \textbf{C1}, \textbf{C2}, \textbf{C3} and \textbf{C4}. Thus bound of $\lvert a_5 \rvert$ follows from (\ref{new2}).

      Let $H(z)=z+a_2 z^2 +a_3 z^3+\cdots\in \mathcal{S}$ be given by
      $$ 1+ \frac{z H''(z)}{H'(z)}=\varphi(z^4),$$
      where coefficients of $\varphi(z)$ satisfy the conditions \textbf{C1} to \textbf{C4}. Clearly, $H \in \mathcal{C}(\varphi)$ and for the function $H$, we have $a_2=a_3=a_4=0$ and $a_5=B_1/20.$ Thus bound is sharp for $H.$
\end{proof}
\section*{Declarations}
\subsection*{Funding}
The work of the Surya Giri is supported by University Grant Commission, New Delhi, India,  under UGC-Ref. No. 1112/(CSIR-UGC NET JUNE 2019).
\subsection*{Conflict of interest}
	The authors declare that they have no conflict of interest.
\subsection*{Author Contribution}
    Each author contributed equally to the research and preparation of the manuscript.
\subsection*{Data Availability} Not Applicable.

\end{document}